\documentclass[11pt]{article}
\usepackage[pagebackref,colorlinks=true,pdfpagemode=none,urlcolor=blue,linkcolor=blue,citecolor=blue]{hyperref}

\usepackage{amsthm, amssymb, bm}
\usepackage{soul, color}
\usepackage[]{amsmath}
\usepackage[]{amsfonts}
\usepackage[]{fancyhdr}
\usepackage[]{graphicx}
\usepackage[]{empheq}
\usepackage[toc,page]{appendix}
\newtheorem{theorem}{Theorem}[section]
\newtheorem{lemma}[theorem]{Lemma}
\newtheorem{proposition}[theorem]{Proposition}
\newtheorem{corollary}[theorem]{Corollary}

\newtheorem{remark}[theorem]{Remark}
\newtheorem{hypothesis}[theorem]{Hypothesis}

\makeatletter

\newcommand{\Rmnum}[1]{\expandafter\@slowromancap\romannumeral #1@}
\makeatother

\def \Imm {\mathbb{I}}

\def \Rm {\mathbb{R}}

\def\C{\mathcal{C}}

\renewcommand{\div}{\nabla\cdot}

\newcommand{\bfe}{ {\bf e}}

\newcommand{\bfw}{ {\bf w}}

\newcommand{\tr}{ {\text{tr }}}

\newcommand{\cout}[1]{}

\def \Div {\text{Div}}

\def \i {\boldsymbol\iota}

\title{Reconstruction of complex-valued tensors in the Maxwell system from knowledge of internal magnetic fields}
\author{Chenxi Guo\thanks{Department of Applied Physics and Applied Mathematics, Columbia University,  New York NY, 10027; cg2597@columbia.edu} \and Guillaume Bal\thanks{Department of Applied Physics and Applied Mathematics, Columbia University,  New York NY, 10027; gb2030@columbia.edu} }

\begin{document}
\maketitle
\begin{abstract}This paper concerns the reconstruction of a complex-valued anisotropic tensor $\gamma=\sigma+\i\omega\varepsilon$ from knowledge of several internal magnetic fields $H$, where $H$ satisfies the anisotropic Maxwell system on a bounded domain with prescribed boundary conditions. We show that $\gamma$ can be uniquely reconstructed with a loss of two derivatives from errors in the acquisition of $H$. A minimum number of $6$ such functionals is sufficient to obtain a local reconstruction of $\gamma$. In the special case where $\gamma$ is close to a scalar tensor, boundary conditions are chosen by means of complex geometric optics (CGO) solutions. For arbitrary symmetric tensors $\gamma$, a Runge approximation property is used to obtain partial results. This problem finds applications in the medical imaging modalities Current Density Imaging and Magnetic Resonance Electrical Impedance Tomography.
\end{abstract}
\section{Introduction}
The electrical properties of biological tissues are characterized by the conductivity $\sigma$ and the permittivity $\epsilon$. We denote the admittivity as $\gamma=\sigma+\i\omega\epsilon$. Traditionally, the admittivity inside an object in sought from boundary measurements, such as in, e.g., Electrical Impedance Tomography (EIT). This leads to an inverse problem known as the Calder\'{o}n inverse problem. Extensive studies have been made on uniqueness and reconstructions methods for this inverse problem \cite{Sylvester1987,Uhlmann2009}. 
But the corresponding stability estimates are of logarithmic type, which results in a low resolution for the reconstructions, see \cite{Alessandrini1986}.
Moreover, well-known obstructions show that the anisotropic admittivities cannot be uniquely reconstructed from boundary measurements, see \cite{Kohn1984,Uhlmann2009}.
Inverse boundary value problems in electrodynamics have been studied in several papers. In Somersalo {\it et al} \cite{Somersalo1992}, the linearization about constant electromagnetic parameters is studied and a local uniqueness theorem was provided in \cite{Sun1992}. The global uniqueness result was due to Ola, P\"{a}iv\"{a}rinta and Somersalo in \cite{Ola1993} and was improved in \cite{Ola1996}. IBVP with partial data was studied by Calo, Ola and Salo in \cite{Caro2009}. An anisotropic case was studied by Kenig, Salo and Uhlmann in \cite{Kenig2011}.

To remedy the poor resolution of the aforementioned boundary value inverse problems, several recent imaging modalities, called coupled-physics modalities or hybrid imaging modalities aim to couple a high-resolution modality with a high-contrast modality. The inversion of the high-resolution modality from available boundary measurements typically provides internal functionals of the coefficients, which greatly improve the resolution of quantitative reconstructions.
For different types of internal functionals such as current densities and power densities, we refer the reader to \cite{Ammari2008,Bal2012c,Bal2010d,Bal2010,Ider1998,Kuchment2011a,Kuchment2011,Kwon2002a,Monard2012a,Nachman2009,Nachman2010}.

In this paper, we consider a hybrid inverse problem where, in addition to boundary data, we have access to the internal magnetic field $H$. Internal magnetic fields can be measured using a Magnetic Resonance Imaging (MRI) scanner; see \cite{Ider1997} for the experimental details. In \cite{Seo2012}, assuming that the magnetic field $H$ is measurable, Seo \emph{et al} gave a reconstruction for the conductivity in the isotropic case. This paper generalizes the reconstruction of an arbitrary (symmetric) complex-valued tensor and gives an explicit reconstruction procedure for $\gamma=\sigma+\i\omega\varepsilon$. The explicit reconstructions we propose require that all components of the magnetic field $H$ be measured. This is challenging in many practical settings as it requires a rotation of the domain being imaged or of the MRI scanner. The reconstruction of $\gamma$ from knowledge of only some components of $H$, ideally only one component for the most practical experimental setup, is open at present.

%, we have to measure all components of the magnetic field $H$, although this is more complicated than measuring some partial components of $H$ in practice.
 
Let $X$ be a bounded domain with smooth boundary in $\Rm^3$. The smooth anisotropic electric permittivity, conductivity, and the constant isotropic magnetic permeability are respectively described by $\epsilon(x)$, $\sigma(x)$ and $\mu_0$, where $\epsilon(x)$, $\sigma(x)$ are tensors and $\mu_0$ is a constant scalar, known, coefficient. Let $E$ and $H$ denote the electric and magnetic fields inside the domain $X$ with a harmonic time dependence. Thus $E$ and $H$ solve the following system of Maxwell's equations: 
\begin{align}\label{eq:Maxwell} 
\left\{\begin{array}{lll}
\nabla\times E+\i\omega\mu_0H =0\\
\nabla\times H-\gamma E=0
\end{array}\right.
\end{align}
with the boundary condition 
\begin{align}\label{boundary con}
\nu\times E|_{\partial X}=f.
\end{align}
Here, $\gamma=\sigma+\i\omega\varepsilon$ in $X$, $\nu$ is the exterior unit normal vector on the boundary $\partial X$, with the frequency $\omega >0$ fixed. We assume that $\varepsilon(x)$ and $\sigma(x)$ satisfy the uniform ellipticity condition
\begin{align}\label{positive definite}
  \kappa^{-1}\|\xi\|^2\le \xi\cdot\varepsilon\xi \le \kappa\|\xi\|^2, \quad \kappa^{-1}\|\xi\|^2\le \xi\cdot\sigma\xi \le \kappa\|\xi\|^2,\quad \xi\in \Rm^n,
\end{align}
for some $\kappa>0$.

%[These references are hard to understand. It looks like some correspond to boundary measurements, some to the hybrid inverse problem. These references have to be separated more clearly.]

In the present work, we present an explicit (stable) reconstruction procedure for the anisotropic, complex-valued tensor $\gamma$ from knowledge of a set of (at least $6$) magnetic fields $H_j$ for $1\leq j\leq  J$, where $H_j$ solves \eqref{eq:Maxwell} with prescribed boundary conditions $f_j$. We propose sufficient conditions on $J$ and the choice of $\{f_j\}_{1\leq i\leq J}$ such that the reconstruction of $\gamma$ is unique and satisfies elliptic stability estimates.
 
\section{Statements of the main results}

We first introduce the solution space,
\begin{align*}
H_{\Div}^{s}(X):=\{u\in(H^s(X))^3|\Div(\nu\times u)\in H^{s-\frac{1}{2}}(\partial X)\}
\end{align*}
where $\Div$ denotes the surface divergence (see, e.g., \cite{Colton1992} for the definition). Let $TH_{\Div}^s(\partial X)$ denotes the Sobolev space through the tangential trace mapping acting on $H_{\Div}^s(X)$,
\begin{align*}
TH_{\Div}^{s}(\partial X)=\{f\in (H^{s}(\partial X))^3|\Div f\in H^{s}(\partial X)\}
\end{align*}
They are Hilbert spaces for the norms
\begin{align*}
\|u\|_{H_{\Div}^{s}(X)}&=\|u\|_{(H^s(X))^3}+\|\Div(\nu\times u)\|_{H^{s-\frac{1}{2}}(\partial X)}\\
\|f\|_{TH_{\Div}^{s}(\partial X)}&=\|f\|_{(H^s(\partial X))^3}+\|\Div(f)\|_{H^{s}(\partial X)}.
\end{align*}
The boundary value problem \eqref{eq:Maxwell} admits a unique solution $(E,H)\in H_{\Div}^k(X)\times H_{\Div}^k(X)$ with imposed boundary electric condition $\nu\times E|_{\partial X}=f\in TH_{\Div}^{k-\frac{1}{2}}(\partial X)$ except for a discrete set of magnetic resonance frequencies $\{\omega\}$ when $\sigma=0$; see \cite{Kenig2011}. The solution satisfies
\begin{align}
\|E\|_{H_{\Div}^s(X)}+\|H\|_{H_{\Div}^s(X)}\leq C\|f\|_{TH_{\Div}^{s-\frac{1}{2}}(\partial X)}
\end{align}
We assume that  $\omega$ is not a resonance frequency. 

\subsection{Main hypotheses}

We now list the main hypotheses, which allow us to set up our reconstruction formulas, which are {\em local} in nature: the reconstruction of $\gamma$ at $x_0\in X$ requires the knowledge of $\{H_j(x)\}_{1\leq j\leq J}$ for $x$ only in the vicinity of $x_0$.

The first hypothesis requires the existence of a basis of electric fields which satisfy \eqref{eq:Maxwell}.
\begin{hypothesis} \label{main hypo}
Given Maxwell's equations in form of \eqref{eq:Maxwell} with $\varepsilon$ and $\sigma$ uniformly elliptic, there exist $(f_1,f_2,f_3)\in TH_{\Div}^{\frac{1}{2}}(\partial X)^3$ and a sub-domain $X_0\subset X$, such that the corresponding solutions ${E_1,E_2,E_3}$ satisfy
\begin{align*}
	\inf_{x\in X_0} |\det (E_1,E_2,E_3) \ge c_0| >0.
\end{align*}
\end{hypothesis}
Assuming that ${E_1,E_2,E_3}$ solutions to \eqref{eq:Maxwell} satisfy the Hypothesis \ref{main hypo}, we consider additional solutions $\{E_{3+k}\}^m_{k=1}$ and obtain the linear dependence relations for each additional solution, 
\begin{align}
E_{3+k}=\sum_{i=1}^{3}\lambda^k_i E_i, \quad  1\leq k\leq m.
\label{ln dep}
\end{align}
As shown in \cite{Bal2012e,Guo2012a}, the coefficients $\lambda^k_i$ can be computed as follows:
\begin{align*}
\lambda^k_i= - \frac{\det ( E_1,\overbrace{ E_{3+k}}^i, E_3)}{\det ( E_1,E_2, E_3)}= -\frac{\det (\nabla\times H_1,\overbrace{\nabla\times H_{3+k}}^i,  \nabla\times H_3)}{\det (\nabla\times H_1,\nabla\times H_2, \nabla\times H_3)},
\end{align*}
Therefore these coefficients are computable from magnetic fields. The reconstruction procedures will make use of the matrices $Z_k$ defined by
\begin{align}
    Z_k=\left[Z_{k,1} ,Z_{k,2} ,Z_{k,3}\right],\quad \text{where }\quad Z_{k,i}=\nabla \lambda^k_i, \quad 1\leq k\leq m.
\label{Y Z}
\end{align}
These matrices are also uniquely determined from the known magnetic fields.

The next hypothesis which gives a sufficient condition for a local reconstruction of the anisotropic tensor $\gamma$, is that a sufficiently large number of matrices $Z_k$ satisfies a full-rank condition.
\begin{hypothesis}\label{2 hypo}
Assume that Hypothesis \ref{main hypo} holds for $(E_1,E_2,E_3)$ over $X_0\subset X$. We denote $Y$ as the matrix with columns $Y_1,Y_2,Y_3$, where $Y_i=\nabla\times H_i$, $1\leq i\leq 3$. Then there exist $E_{1},\cdots,E_{J=3+m}$ solutions of Maxwell equations \eqref{eq:Maxwell} and some $X' \subseteq X_0$ such that the space,
\begin{align}\label{W full}
\mathcal{W}=\{(\Omega Z_k Y^T)^{sym}|\Omega\in A_3(\Rm),1\leq k \leq m\}.
\end{align}
has full rank in $S_3(\mathbb{C})$ for all $x\in X'$, where $S_3$ and $A_3$ denote the space of $3\times 3$ symmetric and anti-symmetric matrices, respectively.
\end{hypothesis}

\begin{remark}
\label{rem:hypotheses}Hypotheses \ref{main hypo} and \ref{2 hypo} can be both fulfilled for well-chosen boundary conditions $\{f_i\}_{1\leq i\leq 6}$ in \eqref{boundary con} when $\gamma$ is close to a constant tensor $\gamma_0$. The proof of such a statement can be found in Section \ref{sec const}. For a arbitrary tensor $\gamma$, Hypothesis \ref{main hypo} can be fulfilled locally. If we suppose additionally that $\gamma$ is the $C^{1,\alpha}$ vicinity of $\gamma(x_0)$ on some open domain of $x_0$, then Hypothesis \ref{2 hypo} also holds locally, see Section \ref{arbitrary tensor}.
\end{remark}

\subsection{Uniqueness and stability results}
We denote by $M_n(\mathbb{C})$ the space of $n\times n$ matrices with inner product $\langle A,B\rangle:=\tr(A^*B)$. We assume that Hypotheses \ref{main hypo} and \ref{2 hypo} hold over some $X_0\subset X$ with $J=3+m$ solutions $(E_1,\cdots,E_{3+m})$. In particular, the linear space $\mathcal{W}\subset S_3(\mathbb{C})$ defined in \eqref{W full} is of full rank in $S_3(\mathbb{C})$. We will see that the inner products of $(\gamma^{-1})^*$ with all elements in $\mathcal{W}$ can be calculated from knowledge of  $(H_1,\cdots,H_{3+m})$. Together with the fact that $\mathcal{W}$ is also constructed by the measurements, $\gamma$ can be completely determined by $H_1,\cdots,H_{3+m}$. The reconstruction formulas can be found in Theorem \ref{reconstruct formulas}. This algorithm leads to a unique and stable reconstruction in the sense of the following theorem. 
\begin{theorem}\label{stability}
Suppose that Hypotheses \ref{main hypo} and \ref{2 hypo} hold over some $X_0\subset X$ for two $3+m$-tuples $\{E_i\}_{i=1}^{3+m}$ and $\{E_i\}_{i=1}^{3+m}$, solutions of the Maxwell system \eqref{eq:Maxwell} with the complex tensors $\gamma$ and $\gamma'$ satisfying the uniform ellipticity condition \eqref{positive definite}. Then $\gamma$ can be uniquely reconstructed in $X_0$ with the following stability estimate,
\begin{align}
	\|\tilde\gamma - \tilde\gamma'\|_{W^{s,\infty}(X_0)}\le C \sum_{i=1}^{3+m} \|H_i-H'_i\|_{W^{s+2,\infty}(X)}.
	\label{eq:stability}
\end{align}
 for any integer $s>0$. If $\gamma$ is isotropic or in the vicinity of a constant tensor $\gamma_0$, then $\gamma$ can be reconstructed with $6$ measurements and the above estimate holds on $X_0=X$.
\end{theorem}
\begin{remark}
\label{rem:cgo}For the case $\gamma$ is isotropic, it can be reconstructed via a redundant elliptic equation which is based on the construction of Complex Geometrical Optics solutions(CGOs). The algorithms will be given in Section \ref{CGO iso}.
\end{remark}

%\remark 
\paragraph{Outline} The rest of the paper is organized as follows. Section \ref{reconstruct formulas} derives the reconstruction algorithms for an arbitrary anisotropic tensor. Section \ref{sec const} gives the reconstruction formulas and the proof of Hypothesis \ref{main hypo} and \ref{2 hypo} for a constant tensor. Section \ref{CGO iso} derives the global reconstruction formulas of an isotropic tensor by constructing CGO solutions. Section \ref{sec runge} covers the unique continuation property(UCP) and Runge approximation property for an anisotropic Maxwell system. Section \ref{arbitrary tensor} discusses the question of local reconstructibility of a more general tensor.
\section{Reconstruction approaches} 
\subsection{Preliminary}
\paragraph{Exterior calculus and notations:}Throughout this paper, we will identify vector fields with one-forms via the identification $\bfe_i\equiv \bfe^i$ where $\{\bfe_i\}_{i=1}^n$ and $\{\bfe^i\}_{i=1}^n$ denote bases of $\Rm^n$ and its dual, respectively. In this setting, if $V = V^i \bfe_i$ is a vector field, $d V$ denotes the two-vector field 
\begin{align*}
    dV = \sum_{1\le i<j \le n} (\partial_i V^j - \partial_j V^i) \bfe_i\wedge\bfe_j.
\end{align*}
A two-vector field can be paired with two other vector fields via the formula 
\begin{align*}
    A \wedge B (C,D) = (A\cdot C)(B\cdot D) - (A\cdot D) (B\cdot C), 
\end{align*}
Note also the following well-known identities for $f$ a smooth function and $V$ a smooth vector field, rewritten with the notation above:
\begin{align*}
    d (\nabla f) &= 0, \qquad f\in \C^2(X),    \\
    d (fV) &= \nabla f \wedge V + fdV. 
\end{align*}
\paragraph{Hodge star operator:}For $x\in \Rm^n$, let $\{\bfe_1,\cdots,\bfe_n\}$ and $\{\bfe^1,\cdots,\bfe^n\}$ denote the canonical bases of $T_x\Rm^n$ and its dual $T^*_x\Rm^n$. The Hodge star operator on an $l$-form is defined as the linear extension of
\begin{align}
\star(\bfe^{\alpha_1}\wedge\cdots\wedge \bfe^{\alpha_l})|_x=(\bfe^{\beta_1}\wedge\cdots\wedge \bfe^{\beta_{n-l}})|_x
\end{align}
where $(\beta_1,\cdots,\beta_{n-l})\in \{1,\cdots,n\}^{n-l}$ is chosen such that 
\begin{align}
\{\bfe^{\alpha_1},\cdots,\bfe^{\alpha_l},\bfe^{\beta_1},\cdots,\bfe^{\beta_{n-l}}\}
\end{align}
is a positive base of $T^*_x\Rm^n$. For a $l$-form $\eta$, the Hodge star operator follows,
\begin{align}
\star\star\eta=(-1)^{l(n-l)}\eta
\end{align}
\subsection{Reconstruction algorithms}\label{reconstruct formulas}
For some matrices $A,B \in M_n(\mathbb{C})$, we denote their product $A:B$ by,
\begin{align}
A:B=\tr(AB^T)=\tr(A^TB)
\end{align}
Starting with $3$ solutions $(E_1,E_2,E_3)$ satisfying Hypothesis \ref{main hypo}, we then pick additional magnetic fields $H_{3+k}$. The corresponding electric fields $E_{3+k}$ and $E_1, E_2,E_3$ satisfy the linear dependence relations defined in \eqref{ln dep}. We recall the $3\times 3$ matrices, 
\begin{align*}
Y=[Y_1,Y_2,Y_3], \quad Y_i=\nabla\times H_i
\end{align*}
$\nabla\times H_{3+k}$ satisfies the same linear dependence with $Y_1,Y_2,Y_3$ as $E_{n+k}$, $E_1,\cdots,E_n$. Thus $\lambda^k_i$ defined in \eqref{ln dep} are computable from only knowledge of the magnetic fields(we use implicit summation notation),
\begin{align}\label{ln dep H}
\nabla\times H_{3+k}=\lambda^k_i\nabla\times H_i;\quad \nabla\lambda^k_i:= -\nabla \frac{\det (\nabla\times H_1,\overbrace{\nabla\times H_{3+k}}^i,\nabla\times H_3)}{\det Y},
\end{align}
Now we construct the subspace $\mathcal{W}$ of $S_3(\mathbb{C})$ as denoted in Hypothesis \ref{2 hypo},\begin{align}
\mathcal{W}=\{(\Omega Z_k Y^T)^{sym}|\Omega\in A_3(\Rm),1\leq k \leq m\}.
\end{align}
Denote $(\bfw_1,\cdots,\bfw_6)$ as the natural basis of the $6$ dimensional space $S_3({\mathbb{C}})$. Given $6$ vectors $W_1,\cdots,W_{6}$ in $\mathcal{W}$, for any vector $W\in S_3({\mathbb{C}})$, we define a $(7,1)$ type tensor $\mathcal{N}$ dealing with inner products $\langle W,W_p\rangle$:
\begin{align}\label{cross inner}
\mathcal{N}(W,W_1,\cdots,W_6):=\sum_{p=1}^{6}\langle W,W_p\rangle\left|
\begin{array}{ccc}
\langle W_1,\bfw_1\rangle  & \ldots &\langle W_1,\bfw_6\rangle \\
\vdots      &   & \vdots\\
\overbrace{\bfw_1}^i &\ldots &\overbrace{\bfw_6}^i \\
\vdots      &  & \vdots\\
\langle W_{6},\bfw_1\rangle & \ldots & \langle W_6,\bfw_6\rangle
\end{array}
\right|=F(W_1,\cdots,W_6)W
\end{align}
where $F(W_1,\cdots,W_6):=\det\{\langle W_p,\bfw_q\rangle\}_{1\le p,q\le 6}$. Obviously, $\det\{\langle W_p,\bfw_q\rangle\}_{1\le p,q\le 6}W=0$ if and only if $W_1,\cdots,W_{6}$ are linearly dependent. In other words, $\mathcal{N}(W,W_1,\cdots,W_6)=0$ never vanishes if $W_1,\cdots,W_{6}$ are linearly independent and $W\neq 0$.

We summarize the reconstruction algorithms in the following theorem and show that $\gamma$ can be algebraically reconstructed via Gram-Schmidt procedure and the explicit expression \eqref{cross inner}.
\begin{theorem}\label{explicit rec}
Assume that Hypothesis \ref{main hypo} and \ref{2 hypo} are fulfilled on a sub-domain $X_0\subset X$, then $\gamma$ can be reconstructed on $X_0$ as follows
\begin{align}
\gamma=\overline\det\{\langle W_p,\bfw_q\rangle\}_{1\le p,q\le 6}(\mathcal{N}^{-1}(\bar\gamma^{-1},W_1,\cdots,W_6))^*.
\end{align}
Here, $(\bfw_1,\cdots,\bfw_6)$ denotes the natural basis of $S_3({\mathbb{C}})$ and $\{W_p\}_{1\le p\leq 6}$ are linearly independent matrices, which can be constructed from the matrices $\{(\Omega Z_k Y^T)^{sym}\}_{1\leq k \leq m}$ in $\mathcal{W}$ by the Gram-Schmidt procedure. The inner product of $\bar\gamma^{-1}$ with matrices in $\mathcal{W}$ are given by:
\begin{align}
\langle \bar\gamma^{-1},(\Omega Z_k Y^T)^{sym}\rangle=\tr(\Omega M^T_k),
\end{align}
where $M_k:=\frac{\i}{2}\omega\mu_0\star(H_{3+k}-\lambda_i^k H_i)(\bfe_p,\bfe_q)\bfe_q\otimes\bfe_p$ for $1\leq k\leq m$ and $\star$ denotes the Hodge star operator. Moreover, for any other $\gamma'$ satisfying \eqref{positive definite} and Maxwell system \eqref{eq:Maxwell}, we have the following stability estimate,
\begin{align}
	\|\tilde\gamma - \tilde\gamma'\|_{W^{s,\infty}(X_0)}\le C \sum_{i=1}^{3+m} \|H_i-H'_i\|_{W^{s+2,\infty}(X)},
\end{align}
where $C$ is a constant and $s$ is any integer.
\end{theorem}
\begin{proof}
We rewrite the time-harmonic Maxwell's equations \eqref{eq:Maxwell} in terms of differential forms, 
\begin{align}\label{eq:max diff}
\left\{\begin{array}{lll}
\star dE_i=-\i\omega\mu_0H_i\\
\star dH_i=\gamma E_i.
\end{array}\right.
\end{align}
Here $d$ is the exterior derivative and $\star$ denotes the Hodge star operator. Applying the exterior derivative $d$ to \eqref{ln dep} gives,
\begin{align}
d(\sum_{i=1}^3\lambda_i^kE_i-E_{3+k})=0.
\end{align}
Using the formula $d(fV)=df\wedge V+fdV$ for a scalar function $f$ and  a vector field $V$, we have
\begin{align}
d\lambda_i^k\wedge E_i+\lambda_i^k dE_i=dE_{3+k}.
\end{align}
Applying the Hodge operator to \eqref{eq:max diff} and using the fact that $E_i=\gamma^{-1}\nabla\times H_i$, we obtain the following equation,
\begin{align}
d\lambda_i^k\wedge \gamma^{-1}\nabla\times H_i=\i\omega\mu_0\star(H_{3+k}-\lambda_i^k H_i).
\end{align}
By applying two vector fields $\bfe_p$, $\bfe_p$, $1\le p< q\le n$ to the above 2-form, we obtain,
\begin{align}
(\nabla\lambda_i^k\cdot\bfe_p)(\gamma^{-1}Y_i\cdot\bfe_q)-(\nabla\lambda_i^k\cdot\bfe_q)(\gamma^{-1}Y_i\cdot\bfe_p)=\i\omega\mu_0\star(H_{3+k}-\lambda_i^k H_i)(\bfe_p,\bfe_q)
\end{align}
where $Y_i=\nabla\times H_i$ for $1\leq i\leq 3$. The above equation reads explicitly,
\begin{align}
(\gamma^{-1}Y)_{qi}Z_{k,pi}-Z_{k,qi}(\gamma^{-1}Y)_{pi}=\i\omega\mu_0\star(H_{3+k}-\lambda_i^k H_i)(\bfe_p,\bfe_q)\end{align}
which amounts to the following matrix equation,
\begin{align}
\gamma^{-1}YZ^T-(\gamma^{-1}YZ^T)^T=\i\omega\mu_0\star(H_{3+k}-\lambda_i^k H_i)(\bfe_p,\bfe_q)\bfe_q\otimes\bfe_p.
\end{align}
Since $\gamma$ is symmetric, we pick $\Omega\in A_3(\Rm)$ and calculate its '$:$' product with both sides of the above equation, 
\begin{align}
\langle \bar\gamma^{-1},(\Omega Z_k Y^T)^{sym}\rangle=\gamma^{-1}:(\Omega Z_k Y^T)^{sym}=\tr(\Omega M^T_k)
\end{align}
where $M_k:=\frac{\i}{2}\omega\mu_0\star(H_{3+k}-\lambda_i^k H_i)(\bfe_p,\bfe_q)\bfe_q\otimes\bfe_p$. The stability estimate is clear by inspection of the reconstruction procedure. Two derivatives on $\{H_k\}_{1\leq k\leq 3+m}$ are taken in the reconstructions of the matrices $Z_k$ and one derivative is taken for the reconstructiongs of $M_k$. The Gram-Schmidt procedure preserves errors in the uniform norm. Therefore, we have a total loss of $2$ derivatives in the reconstruction of $\gamma$ as indicated in Theorem \ref{stability}.
\end{proof}
\subsection{Global reconstructions close to constant tensor}\label{sec const}
In this section, we assume that $\gamma$ is in the vicinity of a diagonalizable constant tensor $\gamma_0$. We will construct special solutions, namely plane waves, of the Maxwell's equations \eqref{eq:Maxwell} and demonstrate that Hypothesis \ref{main hypo} and \ref{2 hypo} are fulfilled with these solutions. The following lemma shows that Hypothesis \ref{main hypo} is satisfied in the homogeneous media. 
\begin{lemma}\label{basic sol}
Suppose that the admittivity $\gamma$ is sufficiently close to a constant tensor $\gamma_0$, where the real and imaginary parts of $\gamma_0$ satisfy the uniform ellipticity condition \eqref{positive definite}. Then Hypothesis \ref{main hypo} holds on $X$.
\end{lemma}
\begin{proof}
Decompose the tensor $\gamma_0=Q\Lambda Q^T$ for a diagonal $\Lambda\in M_3$ and $Q^TQ=I$. This decomposition is possible since  a symmetric matrix is diagonalizable if and only if it is complex orthogonally diagonalizable, see \cite[Theorem 4.4.13]{Horn1990}. We write $Q=[\beta_1,\beta_2,\beta_3]$ and $k_1$, $k_2$, $k_3$ the components on the diagonal of $\Lambda$, such that $\gamma_0\beta_j=k_j\beta_j$, $j=1,2,3$. We choose plane waves as possible solutions to Maxwell's equations \eqref{eq:Maxwell},
\begin{align}
E_j=\beta_je^{i\zeta_j\cdot x}, \quad 1\le j\le 3,
\end{align}
with some $\zeta_j$ to be chosen in $\mathbb{C}^3$. Applying the curl operator to the first equation in \eqref{eq:Maxwell}, we get the vector Helmholtz equation,
\begin{align}\label{helm}
\nabla\times\nabla\times E_j+\i\omega\mu_0\gamma_0 E_j=0
\end{align}
where $\gamma_0=\sigma_0+\i\omega\varepsilon_0$. Using the fact that $\nabla\times\nabla\times=-\Delta+\nabla\div$, the above equation amounts to 
\begin{align}
(\zeta_j\cdot\zeta_j)e^{\i\zeta_j\cdot x}\beta_j-(\beta_j\cdot\zeta_j)e^{\i\zeta_j\cdot x}\zeta_j+\i\omega\mu_0e^{\i\zeta_j\cdot x}\gamma_0 \beta_j=0
\end{align}
Since $e^{\i\zeta_j\cdot x}$ is never zero, the above equation reduces to,
\begin{align}\label{check gamma}
(\beta_j\cdot\zeta_j)\zeta_j-(\zeta_j\cdot\zeta_j)\beta_j=\i\omega\mu_0\gamma_0\beta_j.
\end{align} 
By choosing $\zeta_j$ to be orthogonal to $\beta_j$ and $\zeta_j\cdot\zeta_j=-\i\omega\mu_0k_j$, equation \eqref{check gamma} obviously holds by noticing that $\gamma_0\beta_j=k_j\beta_j$. From the above analysis, the solutions can be chosen as follows,
\begin{align}\label{indep sol}
\left\{\begin{array}{lll}
E_1=\beta_1e^{it_1\beta_2\cdot x}\\
E_2=\beta_2e^{it_2\beta_3\cdot x}\\
E_3=\beta_3e^{it_3\beta_1\cdot x}
\end{array}\right.
\end{align}
where $t_i$ are chosen such that $t_i^2=-\i\omega\mu_0k_i$ for $1\leq i\leq 3$. Then $E_1,E_2,E_3$ are solutions to Maxwell's equations \eqref{eq:Maxwell} and are obviously independent.
\end{proof}
The next proposition states that, some proper linear combinations of the solutions chosen in Hypothesis \ref{main hypo} also satisfy the Maxwell system \eqref{eq:Maxwell}. 
\begin{proposition}\label{add sol}
Let us choose the electric fields $E_{3+k}=\sum_{i=1}^{3}\lambda^k_i \beta_ie^{\i\zeta_i\cdot x}$ such that  $\lambda^k$ has a constant gradient verifying that $\nabla\lambda^k_i\perp\{\beta_i, \zeta_i\}$, where $\beta_i, \zeta_i$ are chosen in \eqref{indep sol}.
Then $E_{3+k}$ solves Maxwell's equations \eqref{eq:Maxwell} for $\gamma=\gamma_0$.
\end{proposition} 
\begin{proof}
Assume that Hypothesis \ref{main hypo} holds and pick $E_i=\beta_ie^{\i\zeta_i\cdot x}$ defined in \eqref{indep sol} for $i=1,2,3$. We pick addition electric fields as indicated in \eqref{ln dep},
\begin{align}
E_{3+k}=\sum_{i=1}^{3}\lambda^k_i E_i, \quad  k=1,2,\ldots
\end{align}
where $\lambda_i^k$ are to be determined. Inserting $E_{n+k}$ into the vector Helmholtz equation \eqref{helm}, we get, 
\begin{align*}
\nabla\times\nabla\times E_{n+k}=&\nabla\times\nabla\times (\lambda^k_i E_i)\\
=&(\div E_i+E_i\cdot\nabla )\lambda^k_i - (\div\nabla\lambda^k_i+\nabla\lambda^k_i\cdot\nabla)E_i +\nabla\lambda^k_i\times\nabla\times E_i+\lambda^k_i\nabla\times\nabla\times E_i\\
=&-\i\omega\mu_0\gamma_0 \lambda^k_i E_i
\end{align*}
Here we choose $\nabla\lambda^k_i$ to be constant and $\nabla\lambda^k_i\perp \beta_i$. Using the fact that $\div E_i=0$ for the special solutions in \eqref{indep sol} and $E_i$ satisfies the Helmholtz equation \eqref{helm}, the above equation reads
\begin{align}\label{E i}
-(\nabla\lambda^k_i\cdot\nabla)E_i+\nabla\lambda^k_i\times(\nabla\times E_i)=0.
\end{align}
Let $\nabla_{E_i}$ denotes the subscripted gradient operator on the factor $E_i$, the basic formulas for curl operator give that,
\begin{align*}
\nabla\lambda^k_i\times(\nabla\times E_i)=&\nabla_{E_i}(\nabla\lambda^k_i\cdot E_i)- (\nabla\lambda^k_i\cdot\nabla)E_i\\
=&\i(\nabla\lambda^k_i\cdot\beta_i)(\nabla\lambda^k_i\cdot\bfe_p) e^{\i\zeta_i\cdot x}\bfe_p-(\nabla\lambda^k_i\cdot\nabla)E_i.
\end{align*}
By choosing $\nabla\lambda^k_i\perp \beta_i$, equation \eqref{E i} reduces to,
\begin{align}
(\nabla\lambda^k_i\cdot\nabla)E_i=\i(\nabla\lambda^k_i\cdot\zeta_i)E_i=0.
\end{align}
Since $E_1,E_2,E_3$ are independent, the above equation holds if and only if $\nabla\lambda^k_i\cdot\zeta_i=0$, for $i=1,2,3$. Therefore, $E_{3+k}=\sum_{i=1}^{3}\lambda^k_i \beta_ie^{\i\zeta_i\cdot x}$ solves the Maxwell's equation \eqref{eq:Maxwell}, with $\nabla\lambda^k_i, \beta_i, \zeta_i$ an orthogonal basis in $\mathbb{C}^3$. 
\end{proof}

Thanks to Proposition \ref{add sol}, we can choose 3 additional solutions as follows:
\begin{align}\label{choose add sol}
\left\{\begin{array}{lll}
E_{3+1}=\lambda_1E_1=\lambda_1\beta_1e^{it_1\beta_2\cdot x}\\
E_{3+2}=\lambda_2E_2=\lambda_2\beta_2e^{it_2\beta_3\cdot x}\\
E_{3+3}=\lambda_3E_3=\lambda_3\beta_3e^{it_3\beta_1\cdot x}
\end{array}\right.
\end{align}
where $E_1,E_2,E_3$ are chosen in \eqref{indep sol} and $\nabla\lambda_1,\nabla\lambda_2,\nabla\lambda_3$ are chosen to be $\beta_3,\beta_1,\beta_2$, respectively.\\

The following lemma proves that $\mathcal{W}$ is of full rank in $S_3(\mathbb{C})$ in homogeneous media. 
\begin{lemma}\label{sol full rank}
Suppose that the admittivity $\gamma$ is sufficiently close to a constant tensor $\gamma_0$. Then Hypothesis \ref{2 hypo} is fulfilled by choosing a minimum number of 6 electric fields as indicated in \eqref{indep sol} and \eqref{choose add sol}.
\end{lemma}
\begin{proof}
As indicated in Proposition \ref{add sol}, we pick additional solutions $E_{n+k}=\lambda_kE_k$, for $k=1,2,3$, where $\nabla\lambda_1=\beta_3$,$\nabla\lambda_2=\beta_1$ and $\nabla\lambda_3=\beta_2$. Let $A\in S_3(\mathbb{C})$ and suppose that $A\perp\mathcal{W}$, we aim to show that $A$ vanishes. Decompose $A$ in terms of $\beta_i\otimes\beta_j$,
\begin{align}
A=A_{ij}\beta_i\otimes\beta_j, \quad \text{where} \quad A_{ij}=A_{ji}.
\end{align}
Here and below, we use the implicit summation notation for the index $i$ and $j$. Thus,
\begin{align*}
  Z_k Y^T=Z_k(\gamma E)^T=&-\frac{1}{\i\omega\mu_0}Z_k[(\zeta_1\cdot\zeta_1)E_1,(\zeta_2\cdot\zeta_2)E_2,(\zeta_3\cdot\zeta_3)E_3]^T\\
=&-\frac{1}{\i\omega\mu_0}(\zeta_k\cdot\zeta_k)\nabla\lambda_k\otimes E_k
\end{align*}
for $k=1,2,3$. Since $A\perp\mathcal{W}$ implies that $Z_k Y^TA$ is symmetric, we deduce the following equation,
\begin{align}
A_{ij}(\nabla\lambda_k\otimes E_k)(\beta_i\otimes\beta_j)=A_{ij}(\beta_i\otimes\beta_j)(E_k\otimes\nabla\lambda_k).
\end{align}
By definition $E_k=\beta_k e^{\i\zeta_k\cdot x}$ and the orthogonality of $\{\beta_i\}_{1\leq i\leq 3}$ , the above equation reduces to
\begin{align}
A_{i,k+1}(\beta_k\otimes\beta_i-\beta_i\otimes\beta_k)=0, \quad \text{for}\quad k=1,2,3
\end{align}
where we identify $k+1:=1$, for $k=3$. Notice that $\{\beta_k\otimes\beta_i-\beta_i\otimes\beta_k\}_{i,k=1,2,3}$ form a basis in $A_3(\mathbb{C})$, so obviously $A_{i,k+1}=0$, for any $i\neq k$, which implies that $A_{ij}=0$, for any $i,j$. Thus $\mathcal{W}$ is of full rank in $S_3(\mathbb{C})$.
\end{proof}
\begin{remark}
Since the Maxwell system can be written in the sense of differential forms as in \eqref{eq:max diff} for an arbitrary $n$ dimension space, the above reconstruction formulas can thus be generalized to the $n$ dimensional case. The proof of Lemmas \ref{basic sol} and \ref{sol full rank} in $n$ dimensions is analogous to the $3$ dimensional case.  
\end{remark}
\subsection{Global reconstructions for isotropic tensor}\label{CGO iso}
In this section, we suppose that the admittivity $\gamma$ is scalar. We will show that $\gamma$ can be reconstructed via a redundant elliptic system by constructing $6$ Complex Geometrical Optics solutions. CGO solutions are constructed in \cite{Colton1992a} and their properties can be extended to higher order Sobolev spaces, see \cite{Chen2013}. The approach in \cite{Chen2013} can be used to reconstruct the scalar $\gamma$.
\begin{theorem}
Let $\gamma(x)$ be a smooth scalar function. Then there exist $6$ internal magnetic fields $\{H_i\}_{1\leq i\leq 6}$ such that $\gamma$ is uniquely reconstructed via the following redundant elliptic equation,
\begin{align}\label{re ellip}
\nabla\gamma+\beta(x)\gamma=0
\end{align}
where $\beta(x)$ is an invertible matrix, which is uniquely determined by the measurements. Moreover, the stability result \eqref{eq:stability} holds for $X_0=X$.
\end{theorem}
\begin{proof}
The system \eqref{eq:Maxwell} can be rewritten as the Helmholtz equation,
\begin{align}\label{eq:helm}
\nabla\times\nabla\times E-k^2nE=0
\end{align}
where the wave number $k$ is given by $k=\omega\sqrt{\epsilon_0\mu_0}$ with $\epsilon_0$ the dielectric constant, and the refractive index $n=\frac{1}{\epsilon_0}(\epsilon(x)-\i\frac{\sigma(x)}{\omega})$. The proof is based on the construction of complex geometrical optics solutions of the form,
\begin{align}\label{CGO sol}
E(x)=e^{\i\zeta\cdot x}(\eta+R_{\zeta}(x))
\end{align}
where $\zeta,\eta\in \mathbb{C}^3$, $\zeta\cdot\zeta=k^2$ and $\zeta\cdot\eta=0$. The existence of $R_{\zeta}$ in $\C^2(X)$ was proved in \cite{Colton1992a} and can be generalized to an arbitrary regular space $\C^d(X)$, see \cite{Chen2013}. 
Now picking two CGO solutions $E_1,E_2$ as defined in \eqref{CGO sol}, we derive the following equation from \eqref{eq:helm},
\begin{align}
\nabla\times\nabla \times E_1\cdot E_2+\nabla\times\nabla\times E_2\cdot E_1=0
\end{align}
Substituting the measurements $Y_j=\nabla\times H_j=\gamma E_j$ into the above equation gives the following transport equation,
\begin{align}\label{transport theta}
\theta\cdot\nabla\gamma+\vartheta\gamma=0
\end{align}
where 
\begin{align}
\begin{split}
\theta &=\chi[(Y_2\cdot\nabla)Y_1+(\nabla\cdot Y_1)Y_2+2\nabla_{Y_2}(Y_1\cdot Y_2)-(Y_1\cdot\nabla)Y_2-(\nabla\cdot Y_2)Y_1-2\nabla_{Y_1}(Y_1\cdot Y_2)],\\
\vartheta &=\chi(\nabla\times\nabla \times Y_1\cdot Y_2-\nabla\times\nabla \times Y_2\cdot Y_1).
\end{split}
\label{theta}
\end{align}

We choose two specific sets of vectors $\zeta,\eta$ as in \cite{Colton1992a}. Define $\zeta_1,\zeta_2$ and $\eta_1,\eta_2$ in terms of a large real parameter $c$ and an arbitrary real number $a$,
\begin{align}\label{two CGO}
\begin{split}
\left\{\begin{array}{lll}
\zeta_1&=(a/2,\i\sqrt{c^2+a^2/4-k^2},c),\\
\zeta_2&=(a/2,-\i\sqrt{c^2+a^2/4-k^2},-c)
\end{array}\right.,\quad
\left\{\begin{array}{lll}
\eta_1&=\frac{1}{\sqrt{c^2+a^2}}(c,0,-a/2)\\
\eta_2&=\frac{1}{\sqrt{c^2+a^2}}(c,0,a/2).
\end{array}\right.
\end{split}
\end{align}
Note that
\begin{align}
\begin{split}
\lim_{c\rightarrow\infty}\eta_j&=\eta_0:=(1,0,0), \quad j=1,2,\\
\lim_{c\rightarrow\infty}\frac{\zeta_1}{|\zeta_1|}&=-\lim_{c\rightarrow\infty}\frac{\zeta_2}{|\zeta_2|}=\zeta_0:=\frac{1}{\sqrt{2}}(0,\i,1), 
\end{split}
\end{align}
and 
\begin{align}
\zeta_1+\zeta_2=(a,0,0), \quad \zeta_0\cdot\zeta_0=0,\quad \eta_0\cdot\zeta_0=0
\end{align}
By choosing $\chi(x)=-e^{-\i(\zeta_1+\zeta_2)\cdot x}\frac{1}{4\sqrt{2}c}$, $\theta$ and $\zeta_0$ have approximately the same direction when $|\zeta|$, the length of $\zeta_1,\zeta_2$, tends to infinity (see \cite[Proposition 3.6]{Chen2013}),
\begin{align}\label{large tau}
\|\theta-\gamma^2\zeta_0\|_{\C^d(X)}\leq \frac{C}{|\zeta|}.
\end{align}
Now we choose 3 independent unit vectors $\zeta_0^j$ and $\eta_0^j$, such that $\zeta_0^j\cdot\zeta_0^j=\zeta_0^j\cdot\eta_0^j=0$, $j=1,2,3$. Similarly to \eqref{two CGO}, we choose $(\zeta_1^j,\zeta_2^j)$ and $(\eta_1^j,\eta_2^j)$ such that, $|\zeta|:=|\zeta_1^j|=|\zeta_2^j|$, and also,
\begin{align}
\lim_{|\zeta|\rightarrow\infty}\frac{\zeta^j_1}{|\zeta|}=-\lim_{|\zeta|\rightarrow\infty}\frac{\zeta^j_2}{|\zeta|}=\zeta^j_0 \quad \text{and}~ \lim_{|\zeta|\rightarrow\infty}\eta^j=\eta^j_0.
\end{align}
We pick 3 pairs of CGO solutions $\{E^j_1,E^j_2\}_{1\leq j\leq 3}$ as defined in \eqref{CGO sol} and define the corresponding $\{\theta_j,\vartheta_j\}_{1\leq j\leq 3}$ by \eqref{theta}. From the estimate \eqref{large tau}, we deduce that $[\theta_1,\theta_2,\theta_3]$ is invertible for $|\zeta|$ sufficiently large. Therefore equation \eqref{transport theta} amounts to a redundant elliptic equation,
\begin{align}\label{redundant ellip}
\nabla\gamma+\beta(x)\gamma=0,
\end{align}
where $\beta=[\theta_1,\theta_2,\theta_3]^{-1}[\vartheta_1,\vartheta_2,\vartheta_3]$. Then $\gamma$ can be reconstructed using \eqref{redundant ellip} if it is known at one point on the boundary. Since we have to differentiate the measurements twice for the construction of $\beta$, there is a loss of two derivatives compared to $H$ for the reconstruction of $\gamma$ via \eqref{redundant ellip}. The stability estimate \eqref{eq:stability} obviously follows.
\end{proof}
\subsection{Runge approximation for the anisotropic Maxwell system}\label{sec runge}
To derive local reconstruction formulas for a more general $\gamma$, we need to control the local behavior of solutions by well-chosen boundary conditions. This is done by means of a Runge approximation. In this section, we will prove the Runge approximation for an anisotropic Maxwell system using the unique continuation property. For UCP and Runge approximation in our context, we refer the readers to, e.g., \cite{Lax1956, Nakamura2005}. 
\subsubsection{Unique continuation property}
Unique continuation property for an anisotropic Maxwell system with only real magnetic permeability $\epsilon$ has been proved in \cite{Eller2006}. We generalize the result to the case of a complex tensor $\gamma=\sigma+\i\omega\epsilon$ in \eqref{eq:Maxwell}. We recall the div-curl system as follows,
\begin{align}\label{div-curl}
\begin{split}
&\gamma(x)E(x)-\nabla\times H(x)=0, \quad \i\omega\mu(x)H(x)+\nabla\times E(x)=0, \\
&\div(\gamma(x)E(x))=0, \qquad \div(\mu(x)H(x))=0
\end{split}
\end{align}
We will use the Calder\'{o}n approach to derive a Carleman estimate which implies the unique continuation property across every $\C^2$-surface. For Calder\'{o}n approach, we refer the readers to \cite{Calderon1958, Nirenberg1973}.

\begin{lemma}[Basic Carleman inequality]\label{Basic carleman}
Let $(u(x,t),v(x,t))\in \C^1(B_{r}(x_0))^3$ with support contained in $|x|\leq r$, $0\leq t\leq T$. There is a constant $C$ independent of $(u,v)$ such that for $r$, $T$ and $k^{-1}$ sufficiently small, the following inequality holds 
\begin{align}\label{carleman main}
\int_0^T\|u,v\|w(t)dt \leq C(k^{-1}+T^2)\int_0^T \|P(u,v)\| w(t)dt.
\end{align}
where $P$ denotes the div-curl operator, 
\begin{align}\label{curl div op}
P(u,v)=(\i\omega\mu v+\nabla\times u,\gamma u-\nabla\times v,\div (\gamma u),\div (\mu v))
\end{align}
Here $\|\cdot\|$ denotes the $L^2$ norm with respect to x-variable, $w(t)=e^{k(T-t)^2}$ with $k$ a positive constant. Then if $(E,H)$ is a solution of the system \eqref{div-curl} in a neighborhood of the origin, vanishing identically for $t<0$, then $(E,H)=0$ in a full neighborhood of the origin. 
\end{lemma}
\begin{proof}
We first introduce the div-curl system, 
\begin{align}\label{eq:sys div}
L(x,D)=(\nabla\times u, \div(\gamma u))
\end{align}
where the principle symbol of $L$ is 
\begin{align}\label{L sys}
L(x,\xi)=
\left(\begin{array}{ccc}
0&-\xi_3&\xi_2\\
\xi_3&0&\xi_1\\
-\xi_2&\xi_1&0\\
\sum_{j=1}^3 \gamma_{1j}\xi_j&\sum_{j=1}^3 \gamma_{2j}\xi_j&\sum_{j=1}^3 \gamma_{3j}\xi_j
\end{array}\right)
\end{align}
Notice that the third curl equation does not involve any derivatives in $x_3$ direction, thus it can be dropped. Then we derive a square system, 
\begin{align}
\tilde{L}(x,\xi)=\left(\begin{array}{ccc}
\xi_3&0&\xi_1\\
0&\xi_3&-\xi_2\\
\gamma_{1j}\xi_j&\gamma_{2j}\xi_j&\gamma_{3j}\xi_j
\end{array}\right)
\end{align}
We rewrite the principal part of \eqref{eq:sys div} in the form $l(x,e_3)D_3u+\bar{L}(x,D')u$, where
\begin{align}
l(x,e_3)=\left(\begin{array}{ccc}
1&0&0\\
0&1&0\\
\gamma_{13}&\gamma_{23}&\gamma_{33}
\end{array}\right)
\end{align}
is invertible and $\bar{L}(x,D')u$ contains only the derivatives with respect to $x_1$ and $x_2$. Hence the equation \eqref{eq:sys div} can be rewritten as follows,
\begin{align}
D_3u+l^{-1}(x,e_3)\bar{L}(x,D')u=l^{-1}(x,e_3)\tilde{L}(x,D)u.
\end{align}
We then calculate the eigenvalues of $l^{-1}(x,e_3)\bar{L}(x,\xi')u$, namely the roots of $p(x,\xi', \alpha)=\det(\alpha\Imm+l^{-1}(x,e_3)\bar{L}(x,\xi'))$. We first list the standard hypotheses in Calder\'{o}n's approach: For $(x,t)$ in a neighborhood of the origin, and for every unit vector $\xi'$ in $\Rm^n$:
\begin{itemize}
  \item  $p(x,\xi', \alpha)$ has at most simple real roots $\alpha$ and at most double complex roots,
  \item distinct roots $\alpha_1$, $\alpha_2$ satisfy $\|\alpha_1-\alpha_2\|\geqq \epsilon >0$
  \item nonreal roots $\alpha$ satisfy $\|\Im\alpha\|\geqq \epsilon $
\end{itemize}
Here $\epsilon$ is some fixed positive constant. In the following, the summations will be from $1$ to $2$. 
\begin{align*}
p(x,\xi', \alpha)&=\det(l(x,e_3))^{-1}\det\tilde{L}(x,\xi',\alpha)\\
  &=\frac{1}{\gamma_{33}}\alpha(\gamma_{jk}\xi_j\xi_k+2\alpha\gamma_{3j}\xi_j+\gamma_{33}\alpha^2)\\
&=\frac{1}{\gamma_{33}}\alpha(\xi',\alpha)\gamma(\xi',\alpha)^T
\end{align*}
Hence the three roots of $p(x,\xi', \alpha)$ are:
\begin{align}
\alpha_1=0, \qquad \alpha_{2,3}=-\frac{\gamma_{3j}\xi_j}{\gamma_{33}}\pm \sqrt{(\frac{\gamma_{3j}\xi_j}{\gamma_{33}})^2-\frac{\gamma_{jk}\xi_j\xi_k}{\gamma_{33}}}
\end{align}
$\alpha_2$ and $\alpha_3$ satisfy the above hypothesis and the prove is essentially given in \cite[Lemma 17.2.5]{Hormander1994}. Since $\Re\gamma$ and $\Im\gamma$ are both positive definite, the roots $ \alpha_{2,3}$ are non-real, by noticing that $(\xi',\alpha)\gamma(\xi',\alpha)^T\neq 0$ for real $\alpha$. Then $|\alpha_1-\alpha_2|^2=4|(\frac{\gamma_{3j}\xi_j}{\gamma_{33}})^2-\frac{\gamma_{jk}\xi_j\xi_k}{\gamma_{33}}|=\frac{4}{|\gamma_{33}|^2}|(\gamma_{3j}\xi_j)^2-\gamma_{33}\gamma_{jk}\xi_j\xi_k|$. A simple calculation shows that,
\begin{align}
|(\gamma_{3j}\xi_j)^2-\gamma_{33}\gamma_{jk}\xi_j\xi_k| & \geqq |\Im((\gamma_{3j}\xi_j)^2-\gamma_{33}\gamma_{jk}\xi_j\xi_k))|\\
&=|e_3^T\tau e_3\cdot\xi'^T\epsilon\xi'+e_3^T\epsilon e_3\cdot\xi'^T\tau\xi'-2e_3^T\tau\xi'\cdot e_3^T\epsilon\xi'|
\end{align}
which is obviously strictly positive for a unit vector $\xi'$ by the Cauchy-Schwarz inequality and the fact that $e_3=(0,0,1)$ and $\xi'=(\xi_1,\xi_2,0)$ are not collinear. Then we obtain a Carleman type inequality, see \cite[Page 33]{Nirenberg1973},
\begin{align}\label{carl 1}
\int_0^T\|u\|w(t)dt \leq C(k^{-1}+T^2)\int_0^T (\|\nabla\times u)\|+\|\div (\gamma u)\| )w(t)dt.
\end{align}
Here $u$ is compactly support in a neighborhood of the origin and $k^{-1}$ and $T$ are sufficiently small. Applying the same analysis to $v$ and using Cauchy-Schwarz, we have the following estimate,
\begin{align}\label{carl 2}
\int_0^T\|u,v\|w(t)dt & \leq C(k^{-1}+T^2)[\int_0^T (\| \i\omega\mu v+\nabla\times u)\|+\|\gamma u-\nabla\times v\|)w(t)dt\\
&+\int_0^T(\|u,v\|+\|\div (\gamma u)\|+ \|\div (\mu v)\|)w(t)dt].
\end{align}
The term $\|u,v\|$ can be moved to the RHS by choosing $k^{-1}$ and $T$ sufficiently small. We thus get the Carleman estimate,
\begin{align}
\int_0^T\|u,v\|w(t)dt \leq C(k^{-1}+T^2)\int_0^T \|P(u,v)\| w(t)dt,
\end{align}
where $P$ denote the div-curl operator in \eqref{curl div op}. Now suppose $z=(E,H)$ satisfies $Pz=0$. Let $\zeta(t)$ be a nonnegative smooth function defined in $t\geqq 0$ equal to 1 for $t\leq 2T/3$ and $0$ for $t\geqq T$. By applying \eqref{carleman main} to $(u,v)=\zeta z$, we obtain that,
\begin{align}
\int_0^{\frac{2T}{3}}\|z\|^2wdt\leq C(k^{-1}+T^2)\int_{\frac{2T}{3}}^T \|P(\zeta z)\|^2wdt \leq  C'(k^{-1}+T^2)\int_{\frac{2T}{3}}^T wdt
\end{align}
with some fixed constant $T$ and $C'$ independent of $k$. Thus, we obtain,
\begin{align}
e^{kT^2/4}\int_0^{\frac{T}{2}}\|z\|^2dt\leq C'(k^{-1}+T^2)Te^{kT^2/9}
\end{align}
Letting $k\rightarrow\infty$, we see that $z=0$ for $t\leq T/2$.
\end{proof}
Due to the above lemma, we may generalize the unique continuation property to the Maxwell system with a complex tensor $\gamma$, which is a more general case of \cite[Corollary 1.2]{Eller2006} but requires more smoothness of the coefficients in order to apply the Calder\'on machinery. We formulate it as the following theorem.
\begin{theorem}\label{UCP theorem}
Let $(E,H)\in H^1(X)$ satisfying the Maxwell's equation \eqref{eq:Maxwell} and let $S=\{\Phi(x)=\Phi(x_0)\}$ be a level surface of the function $\Phi\in \C^2(\bar{X})$ near $x_0\in X$ such that $\nabla \Phi(x_0)\neq 0$. If $(E,H)$ vanish on one side of $S$, then $(E,H)=0$ in a full neighborhood of $x_0\in X$.
\end{theorem}
\begin{proof}
The proof is analogue to Lemma \ref{Basic carleman} by introducting new coordinates $x_3=\Phi(x)-\Phi(x_0)$, in which the level surfaces of $\Phi$ becomes $\{x_3=0\}$. By the ellipticity property of Maxwell's equations, the analysis of the new system can be returned to the original one. See, for example, \cite{Eller2006} for details.
\end{proof}
\subsubsection{Proof of Runge approximation property}
The Runge approximation can be proved with the unique continuation property of Maxwell's system, since we have the uniqueness of the Cauchy problem near every direction. The prove of the following theorem follows the idea in \cite{Nakamura2005}.
\begin{theorem}[Runge approximation]\label{RAP}
Let $X_0$ and $X$ be two bounded domains with smooth boundary such that $\bar{X_0}\subset X$. Let $(E_0,H_0)\in H^1(X_0)$ locally satisfy the Maxwell's equations \eqref{div-curl},
\begin{align}
P(E,H)=0 \quad (X_0) .
\end{align} 
Then for each $\epsilon>0$, there is a function $f_{\epsilon}\in TH_{\Div}^{\frac{1}{2}}(\partial X)$ such that the solutions $(E_{\epsilon},H_{\epsilon})\in H^1(X)^3$ satisfy,
\begin{align}
P(E_{\epsilon},H_{\epsilon})=0  \quad (X),  \quad \nu\wedge E_{\epsilon}|_{\partial X}=f_{\epsilon}
\end{align}
Moreover, for a compact subset $K\subset X_0$,
\begin{align}
\|E_{\epsilon}-E_0\|_{H^1(K)}\leq \epsilon.
\end{align}
\end{theorem}
\begin{proof}
We rewrite Maxwell's equations \eqref{div-curl} into the following Helmholtz-type equation,
\begin{align}
L(E):=\nabla\times\mu^{-1}\nabla\times E+\i\omega\gamma E=0.
\end{align}
Applying the interior estimate to solutions of Maxwell's equations (see \cite{Weber1981}), we get the local estimate, 
\begin{align}
\|E_{\epsilon}-E_0\|_{H^1(K)}\leq C\|E_{\epsilon}-E_0\|_{L^2(\tilde{K})}
\end{align}
for some constant $C>0$ and where $\tilde K\subset X_0$ is a compact containing $K$. Therefore, we wish to prove that,
\begin{align}
M=\{w:w=u|_{\tilde K}, u\in H^1(X), Lu=0  ~\text{in} ~ X\}
\end{align}
is dense in 
\begin{align}
N=\{w:w=u|_{\tilde K}, u\in H^1(X_0), Lu=0  ~\text{in} ~ X_0\}
\end{align}
for the strong $L^2$ topology. By Hahn Banach theorem, this means that for all $f\in L^2(\tilde K)$ such that,
\begin{align}
(f,w)_{L^2(\tilde K)}=0 \quad \text{for all}~ w~ \text{in}~M
\end{align}
this implies that
\begin{align}
(f,w)_{L^2(\tilde K)}=0 \quad \text{for all}~ w~ \text{in}~N.
\end{align}
We extend $f$ ouside $\tilde K$ and still call it $f$ as the extension on $X_0$. Define then
\begin{align}\label{eq:extension}
L^*E=f ~\text{on}~X, \quad n\wedge E=0 ~\text{on}~ \partial X
\end{align}
where $L^*=\nabla\times\mu^{-1}\nabla\times +\i\omega\gamma^{*} $ denotes the adjoint to $L$. For any $u\in H^1(X)$ satisfying $Lu=0$ on $X$, integrations by parts show that,
\begin{align}
(f,u)_{L^2(\tilde K)}=\int_{\tilde K} f\cdot u^*d\sigma=\int_{X} L^*E\cdot u^* dx=\int_{\partial X} n\wedge (\mu^{-1}\nabla\times E)\cdot u^*d\sigma=0.
\end{align}
Then we deduce that $\nu\wedge (\mu^{-1}\nabla\times E)=0$ on $\partial X$. Combining with equation \eqref{eq:extension}, we obtain,
\begin{align}
L^*E=0 ~\text{on}~ X\backslash\tilde K, \quad \nu\wedge E=\nu\wedge (\mu^{-1}\nabla\times E)=0~ \text{on}~ \partial X.
\end{align}
Recalling that $H=\frac{\i}{\omega}\mu^{-1}\nabla\times E$, we will prove that $(E,H)$ together with all their first order derivatives vanish on $\partial X$, so that the solution can be extended by 0 outside the domain $X$. With a local diffeomorphism, we restrict $\partial X$ on a neighborhood of the plan $x_3=0$ for simplicity. In this particular case, $\nu=e_3$ and $\nu\wedge E=0$ means that,
\begin{align}\label{3rd zero}
E^1=E^2=0 \quad \text{on}~\partial X
\end{align}
where $E^i$ denotes the i-th component of $E$. Moreover, the third component of $\nabla\times E$ vanishes on $\partial X$,
\begin{align}\label{third curl}
\nu\cdot \nabla\times E=\partial_1 E^2-\partial_2 E^1=0 \quad \text{on}~\partial X
\end{align}
by the fact that $\partial_1 E^2-\partial_2 E^1$ concerns only the tangential derivatives of $E^1,E^2$, which vanish on the boundary. As for \eqref{3rd zero}, $\nu\wedge (\mu^{-1}\nabla\times E)=0$ implies that the first and second components of $\mu^{-1}\nabla\times E$ are both zero. Together with \eqref{third curl} and the fact that $\mu^{-1}$ is positive definite, we infer that,
\begin{align}
\nabla\times E=0 \quad \text{on}~\partial X.
\end{align}
Therefore $H=\frac{\i}{\omega}\mu^{-1}\nabla\times E=0$ on $\partial X$. Recalling that $L^*E=\nabla\times\mu^{-1}\nabla\times E+\i\omega\gamma^{*}E=0$ on $X\backslash\tilde K$, we obviously have,
\begin{align}
 \nabla\times H=\gamma^*E \quad \text{on} ~\partial X.
\end{align} 
Since the third component of $\nabla\times H$ only concerns the tangential derivatives, it has to vanish. Then by \eqref{3rd zero} and the fact that $\gamma_{33}\neq 0$, we have the following equality,
\begin{align}
\nabla\times H=E=0 \quad \text{on}~\partial X.
\end{align}
Since the tangential derivatives of $H$ are both zero on the boundary, 
\begin{align}
\nabla\times H=(\partial_2 H^3-\partial_3 H^2,\partial_3 H^1-\partial_1 H^3,\partial_1 H^2-\partial_2 H^1)=0 \quad \text{on}~\partial X
\end{align}
implies that 
\begin{align}\label{H1 H2}
\partial_3 H^1=\partial_3 H^2=0  \quad \text{on}~\partial X.
\end{align}
Noticing that $\div (\mu H)=0$ and $H=0$ on the boundary $\partial X$, we get,
\begin{align}
\sum_{1\leq i,j\leq 3}\partial_i(\mu_{ij}H^j)=\mu_{33}\partial_3 H^3=0  \quad \text{on}~\partial X.
\end{align}
Together with \eqref{H1 H2} and $\mu_{33}\neq 0$, this implies that $\partial_3 H^1=\partial_3 H^2=\partial_3 H^3=0$. Applying the same calculations for $E$ and its first order derivatives as above, we have
\begin{align}
\nabla\times E= \div (\gamma^* E)=0 \quad \text{on}~\partial X.
\end{align}
We deduce that all first-order derivatives of $E$ and $E$ itself vanish on $\partial X$. Thus $(E,H)$ can be extended to $0$ outside $\partial X$. By the unique continuation property in Theorem \ref{UCP theorem}, we conclude that $E=0$ on $X\backslash\tilde K$. So for any $u\in H^1(X_0)$ with $Lu=0$ in $X_0$, we have,
\begin{align}
\int_{\tilde K} f\cdot u^* dx=\int_{X_0} L^*E\cdot u^*=0,
\end{align}
which completes the proof.
\end{proof}
\begin{remark}
\label{rem:UCP} In the above analysis of UCP and Runge approximation, the magnetic permeability $\mu$ in the Maxwell system \eqref{div-curl} can be any uniformly elliptic tensor, but not necessarily a constant scalar $\mu_0$ as imposed at the beginning of this paper. 
\end{remark}
The next corollary shows that the Runge approximation can be applied to more regular spaces, such as H\"{o}lder space. 
\begin{corollary}
Let $X_0\subset X$ be a bounded domain with smooth boundary. With same hypotheses as Theorem \ref{RAP}, there is a open subset $X'\subset X_0$ such that for any $\epsilon$
\begin{align}
\|E_{\epsilon}-E_0\|_{\C^{1,\alpha}(X')}\leq \epsilon
\end{align}
where $E_0$, $E_{\epsilon}$ satisfy the Maxwell equations \eqref{eq:Maxwell} on $X_0$ and $X$, respectively.
\end{corollary}
\begin{proof}
Recall that $E^{\epsilon}$ and $E_0$ satisfy the equations,
\begin{align}
\nabla\times\nabla\times E^{\epsilon}+\i\omega\mu_0\gamma E^{\epsilon}=\nabla\times\nabla\times E_0+\i\omega\mu_0\gamma E_0=0 \quad (X_0).
\end{align}
Let $v=E_{\epsilon}-E_0$, then $v$ also satisfy the equation
\begin{align}\label{v h}
\nabla\times\nabla\times v+\i\omega\mu_0\gamma v=0 \quad (X_0).
\end{align}
Differentiating \eqref{v h} with respect to $x_j$ for $1\leq j\leq 3$, we obtain,
\begin{align}
\nabla\times\nabla\times \partial_j v+\i\omega\mu_0\gamma \partial_j v=-\i\omega\mu_0\partial_j\gamma v \quad (X_0)
\end{align}
where the operator $\partial_j$ denotes the $x_j$-derivative applied on each component of $v$ and $\gamma$. Recalling the local estimate $\|v\|_{H^1(K)}\leq \epsilon$ in Theorem \ref{RAP}, with the interior estimate and the smoothness of $\gamma$, we deduce
\begin{align}
\|\partial_j v\|_{H^1(X')}\leq C \|\partial_j\gamma v\|_{H^1(X')}\leq C'\epsilon
\end{align}
where $X'$ is contained in $K$. We iterate the above procedure such that $s>\frac{5}{2}$. By applying Sobolev embedding theorem, we obtain the following estimate,
\begin{align}
\|v\|_{\C^{1,\alpha}(X')}\leq C\|v\|_{H^s(X')}\leq C''\epsilon
\end{align}
which completes the proof.
\end{proof}
%\remark
\subsection{Local reconstructions with redundant measurements}\label{arbitrary tensor}
In this section, we will show that local reconstructions are possible for a more general $\gamma$ than presented in earlier sections. The linear independence of the matrices in Hypothesis \ref{main hypo} becomes local. If in addition, $\gamma$ is in the $C^{1,\alpha}(X)$ vicinity of a constant tensor $\gamma_0$ on some open domain $X'\subset X$, Hypothesis \ref{2 hypo} also holds locally. We thus need to use potentially more than $6$ internal magnetic fields, although we do not expect this large number of measurements to be necessary in practice. The control of linear independence from the boundary relies on the Runge approximation in Theorem \ref{RAP}. This scheme was used in \cite{Guo2012a, Bal2012}.

\begin{theorem}\label{general tensor}
Let $X\subset \Rm^n$ a smooth domain and $\gamma$ a smooth tensor. Then for any $x_0\in X$, there exists a neighborhood $X'\subset X$ of $x_0$ and $6$ solutions of \eqref{eq:Maxwell} such that Hypothesis \ref{main hypo} holds. Moreover, if $\gamma$ is in the $C^{1,\alpha}$ vicinity of $\gamma(x_0)$, then Hypothesis \ref{2 hypo} also holds locally on some open domain $X_0\subset X$.
\end{theorem}
\begin{proof}
We denote $\gamma_0:=\gamma(x_0)$. We first construct solutions of the constant-coefficient problem by picking the functions $\{E^0_i\}_{1\leq i\leq 6}$ defined in \eqref{indep sol} and \eqref{choose add sol}. These solutions satisfy $\nabla\times\nabla\times E+\i\omega\mu_0\gamma_0 E=0$ and fulfill Hypothesis \ref{main hypo} and \ref{2 hypo} globally. Second, we look for solutions of the form,
\begin{align}
\nabla\times\nabla\times E_i^r+\i\omega\mu_0\gamma E_i^r=0 \quad \text{in}~ B_{r}, \quad \nu\times E_i^r=\nu\times E_i^0 \quad \text{on}~ \partial B_{r}, \quad 1\leq i \leq 6,
\end{align}
where $B_{r}$ is the ball centered at $x_0$ with $r$ to be chosen. Let $w=E_i^r-E_i^0$, 
\begin{align}
\nabla\times\nabla\times w+\i\omega\mu_0\gamma w=\i\omega\mu_0(\gamma_0-\gamma)E_i^0 \quad \text{in}~ B_{r}, \quad \nu\times w=0 \quad \text{on}~ \partial B_{r}.
\end{align}
By the smoothness of $\gamma$ as well as interior regularity results for elliptic equations, we deduce that,
\begin{align}\label{E 0}
\lim_{r\to \infty} \|E_i^r-E_i^0\|_{\C^{0,\alpha}(B_{r})}\leq C\lim_{r\to \infty}\|(\gamma_0-\gamma)E_i^0\|_{\C^{0,\alpha}(B_{r})} =0.
\end{align}
Thus we can fix $r$ sufficiently small such that $\|E_i^r-E_i^0\|_{\C^{0,\alpha}(B_{r})}\leq \epsilon$ for $\epsilon$ sufficiently small. Finally, by the Runge Approximation property, we claims that for every $\epsilon>0$ and $1\leq i\leq 6$, there exists $f_{\epsilon}\in TH_{\Div}^{\frac{1}{2}}(\partial X)$ such that the corresponding solution $E_i^{\epsilon}$ to \eqref{eq:Maxwell} satisfy,
\begin{align}\label{runge ineg}
\|E_i^{\epsilon}-E_i^r\|_{\C^{1,\alpha}(B_{r})}\leq \epsilon, \quad \text{where}~ \nu\times E_i^{\epsilon}=f_{\epsilon} \quad \text{on}~\partial X.
\end{align}
Combined with equation \eqref{E 0}, we deduce that,
\begin{align}
\|E_i^{\epsilon}-E_i^0\|_{\C^{0,\alpha}(B_{r})}\leq 2\epsilon.
\end{align}
By choosing a sufficiently small $\epsilon$, Hypothesis \ref{main hypo} obviously holds by continuity of  the determinant.

In addition, if $\gamma$ is in the $\C^{1,\alpha}$ vicinity of $\gamma_0$, we can choose a sufficiently small $r$, such that \eqref{E 0} holds in $\C^{1,\alpha}(B_r)$,
\begin{align}\label{E 1}
\|E_i^r-E_i^0\|_{\C^{1,\alpha}(B_{r})}\leq C\|(\gamma_0-\gamma)E_i^0\|_{\C^{1,\alpha}(B_{r})} \leq \epsilon
\end{align}
Then together with \eqref{runge ineg}, we derive the estimate as following,
 \begin{align}
\|E_i^{\epsilon}-E_i^0\|_{\C^{1,\alpha}(B_{r})}\leq 2\epsilon
\end{align}
Notice that the space $\mathcal{W}$ constructed in \eqref{W full} contains up to first derivatives of $E_i$. Again by choosing a sufficient small $\epsilon$, the full rank property of $\mathcal{W}$ in Hypothesis \ref{2 hypo} is satisfied by $\{E_i^{\epsilon}\}_{1\leq i\leq 6}$.
\end{proof}

\section*{Acknowledgment} This paper was partially funded by the NSF grant DMS-1108608.

\end{document}